\documentclass{amsart}
\usepackage{graphicx,amscd}
\usepackage[cmtip,all]{xy}

\newtheorem{theorem}{Theorem}[section]
\newtheorem{lemma}[theorem]{Lemma}
\newtheorem{proposition}[theorem]{Proposition}
\newtheorem{corollary}[theorem]{Corollary}
\newtheorem{Conjecture}[theorem]{Conjecture}
\theoremstyle{definition}
\newtheorem{definition}[theorem]{Definition}
\newtheorem{example}[theorem]{Example}

\newtheorem{inductive}[theorem]{Inductive Step}

\theoremstyle{remark}
\newtheorem{remark}[theorem]{Remark}
\newtheorem{notation}[theorem]{Notation}

\numberwithin{equation}{section}

\newcommand{\rar}{\rightarrow}
\newcommand{\lar}{\longrightarrow}
\def\phi{\varphi}
\def\PP{\mathbb P}
\def\A{\mathbb A}
\def\N{\mathbb N}
\def\Z{\mathbb Z}
\def\kx{k[X_0,\ldots,X_n]}
\def\fp{{\mathcal P}}
\def\cc{{\bf C}}
\def\mm{{\bf m}}
\def\gothm{{\bf m}}
\def\cone#1{{\rm cone}\,(#1)}
\def\mincone#1{{\rm Mincone}\,(#1)}

\begin{document}

\title[Resolutions for monomial curves defined by  Arithmetic Sequences]
{Minimal graded Free Resolutions for Monomial Curves defined by Arithmetic Sequences}

 \author{Philippe Gimenez}
 \address{Department of Algebra, Geometry and
Topology, Faculty of Sciences, University of Valladolid, 47005
Valladolid, Spain.}
 \email{pgimenez@agt.uva.es}
 \thanks{The first author is partially supported by MTM2010-20279-C02-02,
{\it Ministerio de Educaci\'on y Ciencia~-~Espa\~na}.}

\author{Indranath Sengupta}
\address{Department of Mathematics, Jadavpur University, Kolkata,
WB 700 032, India.} \email{sengupta.indranath@gmail.com}
\thanks{The second author thanks DST, Government of India for financial support for the
project ``Computational Commutative Algebra", reference no.
SR/S4/MS: 614/09.
}

 \author{Hema Srinivasan}
 \address{Mathematics Department, University of
Missouri, Columbia, MO 65211, USA.}
 \email{SrinivasanH@math.missouri.edu}


\date{}


\begin{abstract}
Let $\mm=(m_0,\ldots,m_n)$ be an arithmetic sequence, i.e., a
sequence of integers $m_0<\cdots<m_n$ with no common factor that
minimally generate the numerical semigroup $\sum_{i=0}^{n}m_i\N$ and
such that $m_i-m_{i-1}=m_{i+1}-m_i$ for all $i\in\{1,\ldots,n-1\}$.
The homogeneous coordinate ring $\Gamma_\mm$ of the affine monomial
curve parametrically defined by $X_0=t^{m_0},\ldots,X_n=t^{m_n}$ is
a graded $R$-module where $R$ is the polynomial ring
$k[X_0,\ldots,X_n]$ with the grading obtained by setting
$\deg{X_i}:=m_i$. In this paper, we construct an explicit minimal
graded free resolution for $\Gamma_\mm$ and show that its Betti
numbers depend only on the value of $m_0$ modulo $n$.  As a consequence, we prove a conjecture of
Herzog and Srinivasan on the eventual periodicity of the Betti numbers of semigroup rings under 
translation for the monomial curves defined by an arithmetic sequence.
\end{abstract}

\maketitle

\section*{Introduction}

The study of affine and projective monomial curves has a long
history beginning with the classification of space monomial curves in \cite{herzog}.
One of the most dramatic results in the subject is the
fact that the number of generators for the defining ideal of these
curves in the affine space $\A^{n+1}$ is unbounded, \cite{bresinski}.  Much of the
study  to date has focussed on determining the generators and the
first Betti number  of the defining ideal  for many different
classes of monomial curves. In this paper, we study the later Betti numbers as
well as the structure of the resolution.   Exact generators in the
case of curves defined by an arithmetic sequence or an almost
arithmetic sequences are known; see \cite {patil}, \cite{malooseng}, 
\cite{lipatroberts}.   In the case of arithmetic sequences, these ideals
have another interesting structure as sum of two determinantal
ideals; see \cite {hip}.  This provides the main impetus for
understanding the resolution of these ideals.  In this article, we
construct the minimal resolution explicitly for these ideals and
compute all the Betti numbers.

\medskip

The main goal  of this article is to prove the following conjecture
that states that in codimension $n$, there are exactly $n$ distinct
paterns for the minimal graded free resolution of a monomial curve
defined by an arithmetic sequence:

  \medskip

\begin{center}
\begin{tabular}{p{11cm}}\it
Curves in affine $(n+1)$-space defined by a monomial parametrization $X_0=t^{m_0}$,
\ldots, $X_n=t^{m_n}$ where $m_0 < \ldots < m_n$ are positive integers in arithmetic
progression have the property that their Betti numbers are determined solely by $m_0$ modulo $n$.
\end{tabular}
\end{center}

\medskip

Basis for this conjecture came from the observations in \cite
{seng}, where an explicit minimal free resolution was constructed
for $n=3$, using Gr\"{o}bner basis  techniques.  Subsequently, this
property was verified to hold for several examples in \cite{eaca},
and it was proved for certain special cases in \cite{hip}, leading
us to frame it as a conjecture. We give a complete proof of the
conjecture in this article.

\medskip

Let $k$ denote an arbitrary field and $R$ be the polynomial ring
$\kx$. Associated to a sequence of positive integers $\mm = (m_0, \ldots, m_n)$,
we have  the $k$-algebra homomorphism $\phi:R\rar k[t]$ given
by $\phi(X_i)=t^{m_i}$ for all $i=0,\ldots,n$.  The ideal
$\fp:=\ker{\phi}\subset R$ is the defining ideal of the monomial
curve in $\A_k^{n+1}$ given by the parametrization $X_0=t^{m_0}$,
\ldots, $X_n=t^{m_n}$ which we denote by ${C}_{\mm}$ . The $k$-algebra of the numerical semigroup
$\Gamma (\mm)$ generated by $\mm = (m_0,\ldots,m_n)$ is the semigroup ring
$k[\Gamma(\mm)]:=k[t^{m_0},\ldots,t^{m_n}]\simeq R/\fp$ which is
one-dimensional.  Moreover, $\fp$ is a perfect ideal of codimension $n$ and it
is well known that it is minimally generated by binomials.

\medskip

Since for any positive integer $t$ the curve ${C}_{t\mm}$  is
isomorphic to the curve $ C_{\mm}$ for they have the same defining
ideal, we may as well assume without loss of generality that the
integers $ m_0, m_1, \ldots, m_n $ have {\it no common factor}.
Further, if the semigroup generated by a proper subset of $\mm$
equals the semigroup $\Gamma (\mm)$, then the curve ${C}_{\mm}$ is
degenerate and
the resolution of its coordinate ring can be studied in a polynomial
ring with less variables. Hence we can reduce to the case where the
integers in $\mm$ {\it minimally generate} the numerical semigroup
$\Gamma (\mm)$.
If, in addition, the integers $m_i$ are in {\it arithmetic
progression}, i.e., $m_i-m_{i-1}=m_{i+1}-m_i$ for all
$i\in\{1,\ldots,n-1\}$, $\mm$ is said to be an {\bf arithmetic
sequence}.

\medskip

An interesting feature that was revealed in \cite{hip} is that when
$\mm = (m_0,\ldots,m_n)$ is  an arithmetic sequence, the ideal $\fp$
can be written as a sum of two determinantal ideals, $\fp=I_2(A) +
I_2(B)$, as we shall recall in Section~\ref{defidealsection}. Here,
$I_2(A)$ is in fact the defining ideal of the rational normal curve
in $\PP^n$.  Let us write $m_0$ as $m_0=an+b$ where $a,b$ are
positive integers and $b\in[1,n]$.  When $b=1$, the sum $I_2(A) +
I_2(B)$ is again a determinantal ideal and its resolution is
described in \cite[Theorem~2.1 \& Corollary~2.3]{hip}.  When $b=n$,
it is just one generator away from a determinantal ideal which is
again simple; see \cite[Theorem~2.4 \& Corollary~2.5]{hip}. The case
$b=2$ corresponds to the case where $k[\Gamma(\mm)]$ is Gorenstein.
The resolution was described in this case when $n=4$ in
\cite[Theorem~2.6]{hip}. This result will be completed in
Section~\ref{secgorenstein} where the resolution will be given for
an arbitrary $n$.

\medskip

The observation that the ideals $I_2(A)$ and $ I_2(B)$ are related
by linear quotients (Lemma~\ref{colon}) holds the key for the
construction of the resolution of $\fp$ in general.  We construct a
tower of mapping cones each of which is a cone over  an inclusion of
a shifted graded Koszul complex into a  graded Eagon-Northcott
complex. Unfortunately,  the above
construction of iterated mapping cone does not yield a minimal free
resolution for $\fp$ and therefore we will have to get rid of the redundancy and make
the resolution minimial. The complete graded description of the
resolution is given in the main Theorem \ref {mainThm}.  As a
consequence, we compute the total Betti numbers   $\beta _j$ in
Theorem~\ref{ThmIndraConj} as follows: if $m_0\equiv b\mod n$ with
$b\in [1,n]$, then
$$
\beta_j=j{n\choose j+1}+
 \left\{\begin{array}{ll}
\displaystyle{(n-b+2-j){n\choose j-1}}&\hbox{if}\ 1\leq j\leq n-b+1,\\
\displaystyle{(j-n+b-1){n\choose j}}&\hbox{if}\ n-b+1<j\leq n.
 \end{array}\right.
$$
These numbers clearly depend only on the reminder of $m_0$ modulo
$n$, as conjectured in \cite[Conjecture~1.2]{hip}.

\medskip

As another application of Theorem~\ref{mainThm}, we prove the following conjecture of Herzog and Srinivasan for monomial curves defined by an arithmetic sequence.  The strong form of the conjecture says that if $\mm$ is any increasing sequence of non negative integers and
$\mm +(j)$ denotes the sequence translated by $j$, then the Betti numbers of the semigroup ring
$k[\Gamma(\mm +(j))]$
are eventually periodic in $j$.
We prove in Theorem \ref{ThmPeriodic} that if $\mm$ is an arithnetic sequence, the strong form of the conjecture holds by showing that the betti numbers are periodic with period $m_n- m_0=nd$ where $d$ is the common difference.

\medskip

We will begin with some preliminaries on the defining ideal of monomial curves associated to arithmetic sequences, followed by some facts from mapping cones that we need in the paper. Section~\ref{secgorenstein} is entirely on Gorenstein monomial curves defined by an arithmetic sequence where we construct the minimal graded free resolution of its coordinate ring as a direct sum of a resoltution and its dual.  Section~\ref{secmain} contains the construction of the minimal graded free resolution of the monomial curves defined by an arithmetic sequence in general.  The last section has consequences of the main theorem (Theorem~\ref{mainThm})  for some relation between the regularity of the semigroup ring and the Frobenius number of the semigroup, an independent proof of the characterization of Gorenstein curves defined by an arithmetic sequence as well as the proof of the periodicity conjecture of Herzog and Srinivasan for the semigroup rings defined by an arithmetic sequence.

\section{Preliminaries}\label{secprelim}

Let $(\mm)= (m_0, \ldots, m_n)$ be an {\it arithmetic sequence},
i.e., a sequence of nonnegative integers such that $m_i=m_0+id$ for
some $d\geq 1$ and all $i\in [0,n]$, and such that $m_0, \ldots, m_n$
are relatively prime and minimally generate the numerical semigroup
$\displaystyle{\Gamma(\mm):=\sum_{0\leq i\leq n}m_i\N}$. Note that
one can always write $m_0$ uniquely as $$m_0=an+b$$ with $a,b$ positive integers and $b\in [1,n]$. The 
integer $a$ is non-zero because the sequence $(\mm)= (m_0, \ldots, m_n)$ is minimal.

\subsection{Defining ideal of monomial curves associated to arithmetic sequences}\label{defidealsection}

One knows by \cite[Theorem~2.1]{hip} that the defining ideal $\fp$
of the affine monomial curve $C_\mm \subset \A_k^{n+1}$ is $I_2(A)+I_2(B)$, the sum of two determinantal
ideals of maximal minors with
$$A= {\left(\begin{array}{ccc}
\begin{array}{c}
X_{0}\\[1.5mm]
X_{1}\\[1.5mm]
\end{array} &
\begin{array}{c}
\cdots\\[1.5mm]
\cdots\\[1.5mm]
\end{array} &
\begin{array}{c}
X_{n-1}\\[1.5mm]
X_{n}\\[1.5mm]
\end{array}
\end{array}\right)},
\ B= {\left(\begin{array}{cccc}
\begin{array}{c}
X_{n}^{a}\\[1.5mm]
X_{0}^{a+d}\\[1.5mm]
\end{array} &
\begin{array}{c}
X_{0}\\[1.5mm]
X_{b}\\[1.5mm]
\end{array} &
\begin{array}{c}
\cdots\\[1.5mm]
\cdots\\[1.5mm]
\end{array} &
\begin{array}{c}
X_{n-b}\\[1.5mm]
X_{n}\\[1.5mm]
\end{array}
\end{array}\right)}.$$
\medskip

It is well-known that the ideal $I_2(A)$ is the defining ideal of
the rational normal curve in $\PP_k^n$ of degree $n$; see, e.g.,
\cite[Proposition~6.1]{eis2}. The fact that $I_2(A)$ is contained in
$\fp$ says that the affine monomial curve $C_\mm$ is lying on the
affine cone over the rational normal curve. Indeed, the following
easy lemma states that arithmetic sequences are precisely the ones
whose associated monomial curve lies on this cone. They are also the
only sequences that make the ideal $I_2(A)$ homogeneous with respect
to the gradation obtained by setting $\deg{X_i}=m_i$ for all $i\in
[0,n]$.

\begin{lemma}\label{equivGrad}
Let $\fp\subset k[X_0,\ldots,X_n]$ be the defining ideal of the
non-degenerate monomial curve $C_\mm \subset \A_k^{n+1}$ associated
to a strictly increasing sequence of integers
$\mm=(m_0,\ldots,m_n)$. The following are equivalent:
\begin{enumerate}
\item
$\exists\, d\in\Z,$ such that $ m_i=m_0+id,\ \forall i\in [0,n]$;
\item
$I_2(A)\subset \fp$;
\item
$I_2(A)$ is homogeneous w.r.t. the weighted gradation on $R$ given
by $\deg{X_i}=m_i$ for all $i\in[0,n]$.
\end{enumerate}
\end{lemma}

\begin{proof}
As we already recalled, if $m_i=m_0+id$ for some $d\geq 1$ then
$I_2(A)\subset \fp$ and $I_2(A)$ is homogeneous w.r.t. the weighted
gradation, so (1) $\Rightarrow$ (2) and (1) $\Rightarrow$ (3).
Conversely, assuming that either (2) or (3) holds, one has that the
integers in the sequence $(\mm)$ satisfy that
$m_i+m_{j+1}=m_{i+1}+m_j$ for all $i,j$ such that $0\leq i<j\leq
n-1$. In particular, for all $j\in [1,n-1]$, one has that
$m_0+m_{j+1}=m_{1}+m_j$, i.e., $m_{j+1}=m_j+d$ if one sets
$d:=m_1-m_0\geq 1$.
\end{proof}

In this paper, homogeneous and graded will mean homogeneous and
graded with respect to the weighted gradation on $R$ given by
$\deg{x_i}=m_i$ for all $i\in[0,n]$. By Lemma~\ref{equivGrad},
$I_2(A)$ is homogeneous and $\fp$ is also homogeneous since
$\fp:=\ker{\phi}$ and the map $\phi:R\rar k[t]$ given by
$\phi(X_i)=t^{m_i}$ is graded of degree 0.

\subsection{The weighted graded version of the Eagon-Northcott
complex}\label{secENgraded}

The minimal resolution of $R/I_2(A)$ is given by the Eagon-Northcott
complex of the matrix $A$ because the height of $I_2(A)$ is $n$.

\medskip

Let $F= \oplus _{i=1}^n R e_i$ be a free $R$-module of rank $n$ with
basis $e_1, \ldots, e_n$ and $G = Rg_1\oplus Rg_2$ be a free
$R$-module of rank 2. The $2\times n$ matrix $A$ represents a map
$\phi: F\lar G^*$. The Eagon-Northcott complex ${\bf E}$ of the
matrix $A$ is
$$
{\bf E}:\ 0\lar E_{n-1} \stackrel{d_{n-1}}{\lar}  E_{n-2}
\stackrel{d_{n-2}}{\lar} \cdots\lar E_1\stackrel{d_{1}}{\lar} E_0\,,
$$
with $E_0=R$ and, for all $i\in [1,n-1]$, $E_i :=\wedge
^{i+1}F\otimes D_{i-1}G$ and $d_i$ is given by the diagonalization
of $\wedge F$ and multiplication of $DG$.

\begin{remark}\label{dualENrmk}
Note that $R/I_2(A)$ is Cohen-Macaulay and the complex ${\bf E }^*$
is exact. Thus, we get
$$
{\bf E}^*:\ 0\lar R   \stackrel {d_1^*}{\lar} E_1^* \lar \cdots
\stackrel { d_{n-2}^* }{\lar}\wedge ^{n-1}F^* \otimes
D_{n-3}G^*\stackrel { d_{n-1}^* }{\lar} \wedge^nF^*\otimes
D_{n-2}G^*.
$$
One has that $d_1^*:R\to \wedge ^2 F^*$ is the map $\wedge ^2
(\phi^*)$, and for $i>1$, the map $d_i^*: \wedge ^iF^*\otimes
D_{i-2}G^* \lar \wedge ^{i+1}F^*\otimes D_{i-1}G^*$ is given by
 $$d_i^*(x\otimes y) = \sum _{t=1}^n x\wedge e_t^* \otimes \phi(e_t)y.$$
After identifying $\wedge ^nF^*$ with $R$ and $\wedge ^{n-1}F^*$
with $F$, we see that $$d_{n-1}^*(e_i \otimes
g_1^{*(r)}g_2^{*(s)})=(-1)^{n-i}[
X_{i-1}g_1^{*(r+1)}g_2^*{(s)}+X_{i}g_1^{*(r)}g_2^{*(s+1)}].
$$
This will be useful in Section~\ref{secgorenstein}.
\end{remark}

The ideal $I_2(A)$ is homogeneous with respect to the usual grading
on $R$ and the Eagon-Northcott complex is indeed a minimal graded
free resolution of $R/I_2(A)$. This minimal graded resolution
 is 2- linear and it is as follows:
 $$
{\bf E}:\ 0\rar R^{n-1}(-n)\stackrel{d_{n-1}}{\lar}
R^{(n-2){n\choose n-1}}(-n+1) \stackrel{d_{n-2}}{\lar} \cdots
\stackrel{d_{s}}{\lar} R^{(s-1){n\choose s}}(-s)
\stackrel{d_{s-1}}{\lar} \cdots
 $$
\begin{flushright}
$\displaystyle{\cdots \stackrel{d_{2}}{\lar} R^{n\choose 2}(-2)
\stackrel{d_{1}}{\lar} R \rar R/I_2(A)\rar0. }$
\end{flushright}
But $I_2(A)$ is also homogeneous with respect to our weighted
gradation on $R$ as observed in Lemma~\ref{equivGrad} and the
Eagon-Northcott complex is also a minimal graded free resolution of
$R/I_2(A)$ with respect to this weighted gradation. Of course,
syzygies are no longer concentrated in one single degree at each
step of the resolution as before. As observed in \cite{hip}, the
successive graded free modules in this resolution are $E_0=R$ and,
for all $s\in [2,n]$,
\begin{eqnarray*}
E_{s-1}&=&\bigoplus_{1\le r_1<\ldots < r_s \le n }
 \left(\bigoplus_{k=1}^{s-1} R(- sm_0 +kd
 -\sum r_i d)\right)\\
&=&
 \bigoplus_{k=1}^{s-1}
 \left(\bigoplus_{1\le r_1<\ldots < r_s \le n} R(- (sm_0 -kd
 +\sum r_i d))\right)\\
&=&
 \bigoplus_{k=1}^{s-1}
 \left(\bigoplus_{0\le r_1<\ldots < r_s \le n-1} R(- (sm_0 +(s-k)d
 +\sum r_i d))\right)\\
&=&
 \bigoplus_{k=1}^{s-1}
 \left(\bigoplus_{0\le r_1<\ldots < r_s \le n-1} R(- (sm_0 +kd
 +\sum r_i d))\right)\\  \,.
\end{eqnarray*}

\begin{notation}
Given two integers $m\geq t\geq 1$, it will be useful to denote by
$\sigma(m,t)$ the collection (with repetitions) of all possible sums of $t$ distinct
nonnegative integers which are all strictly smaller than $m$, i.e.,
$$
\sigma (m,t) = \{ \sum _ {0\le r_1<r_2<\cdots <r_t \le m-1} r_i
\}\,.
$$
For instance, $\sigma (4,2) =
\{1,2,3,3,4,5\}$. Note that for all $t$ and $m$ with $1\leq t\leq
m$, $\#\sigma (m,t)={m\choose t}$.
\end{notation}

\medskip

The weighted graded free resolution of $R/I_2(A)$ given by the
Eagon-Northcott complex can now be written as follows:
 {\small
$$
 0\rar
 \bigoplus_{k=1}^{n-1}R(- (nm_0+kd + {n\choose 2} d ))
 \lar\cdots\lar
 \bigoplus_{k=1}^{s-1}
 \left(\bigoplus_{r\in\sigma(n,s)}
 R(- (sm_0 +kd+rd))\right)\lar\cdots
$$
\begin{flushright}
$ \cdots\lar\displaystyle{
 \bigoplus_{r\in\sigma(n,2)}R(-(2m_0+d+r d))
 \lar R \lar R/I_2(A)\rar 0\,. }$
\end{flushright}
}

\subsection{Mapping cone}\label{subsecMappingCone}

In this section, we will establish our notation for mapping cones,
complexes and some facts on mapping cones that we will need.

\medskip

For any complex ${\bf F} = \oplus F_i$, denote by $(\delta_{\bf
F})_i:F_i\rar F_{i-1}$ the boundary maps of ${\bf F}$, and for any
$t\geq 1$, let ${\bf F^t}$ be the complex whose $i$th term is
$(F^t)_i := F_{i-t}$. Now if $\mu: {\bf F}\rar {\bf G}$ is a map of
complexes, the {\it mapping cone} (or {\it cone}) over $\mu$ is the
complex ${\bf G}\oplus {\bf F^1}$ and it is denoted by $\cone{\mu}$.
The boundary maps of this complex are

$$
\begin{array}{rccl}
 \left(
 \begin{array}{cc}
 (\delta_{\bf G})_i&(-1)^{i}\mu_{i-1}\\
 0&(\delta_{\bf F})_{i-1}
 \end{array}
 \right)
 :&
 G_i\oplus F_{i-1}&\rar &G_{i-1}\oplus
F_{i-2}\,,
 \end{array}
 $$

 \medskip\noindent
i.e., $(\delta _{\cone {\mu}})_i(g_i,f_{i-1})=((\delta_{\bf G})_i(g_i)+(-1)^i \mu
_{i-1}(f_{i-1}),(\delta_{\bf
 F})_{i-1}(f_{i-1}))$.

\bigskip

Let's recall a few  well known  facts on mapping cones that we will
need in the sequel. If ${\bf G}$ is acyclic, i.e., $H_i({\bf G})= 0$
for $i\ge 1$,     then $\cone {\mu}$ is exact  up to degree 1, i.e.,
$H_i(\cone {\mu}) = 0$ for all $i\ge 2$. When ${\bf G}$ is exact
and moreover $\mu_0$ is injective, then $\cone {\mu}$ is acyclic. A
situation of special interest is when $\bf F$ is a resolution of
$R/J$ and $\bf G$ a resolution of $R/I$ for two  ideals $I$ and $J$
in $R$. Then,  given a map of complexes $\mu: {\bf F}\rar {\bf G}$,
$\cone {\mu}$ resoves $R/I+\mu_0(R)$ provided $\mu_0 (J)$ is
contained in $I$. In particular,  consider the following situation:
let $I$ be an ideal in $R$ and take an element $z\in R$. Then,
$$
0\rar R/(I:z) \stackrel {\mu} \lar R/I \lar R/I+(z)\rar 0
$$
is exact, where $\mu$ is the map given by multiplication by $z$. Now if
${\bf F}$ resolves $R/(I:z)$, ${\bf G}$ resolves $R/I$, and $\mu :{\bf
F} \rar {\bf G}$ is a map of complexes induced by $\mu$, then
$\cone{\mu}$ resolves $R/I+(z)$.

\medskip

Let's consider now the graded version of the previous statements. Assume
that $I$ and $J$ are homogeneous ideals in $R$, and consider $\bf F$
a graded resolution of  $R/J$, $\bf G$ a graded resolution of
$R/I$,  and $\mu: {\bf F}\rar {\bf G}$ a graded map of complexes
with $\mu (J)\subset    I$. Then, the exact complex $\cone {\mu}$ is
the graded resolution of $R/I+\mu_0(R)$. In the particular case where $J=(I:z)$ for some homogeneous element $z\in R$ of degree $\delta$, the degree zero map $\mu: R(-\delta)\rar R$ given by multiplication by $z$ induces a graded map of complexes $\mu: {\bf F}(-\delta)\rar {\bf G}$ and $\cone{\mu}$ is a graded free resolution of $R/I+(z)$.
Recall that a resolution
$\bf F$ of $R/I$ is {\it minimal}  if the ranks of the $F_i$'s are
minimal or, equivalently, if $(\delta _{\bf F}({\bf F}))\subset
\gothm{\bf F}$ where $\gothm = (x_0, \ldots, x_n)$ is the irrelevant
maximal ideal. Thus, in a minimal graded resolution, there are no
degree zero components in the resolution unless they are identically
zero.  If $\bf F$ and  $\bf G$ are minimal graded free resolutions of $R/I:z$ and $R/I$ respectively, then
the only possible obstructions for $\cone{\mu}$ to be minimal are
the degree zero components in $\mu$.
In other words, if  $\bf F = \oplus F_i $, $\bf G = \oplus G_i$ with
$F_i = \bigoplus _{j} R(-d_{ij})$ and $G_i = \bigoplus _j
R(-c_{ij})$, and $\mu: {\bf F}\rar {\bf G}$ is the graded map of
complexes induced by  multiplication by $z$, $R(-\delta)\rar R$, then $\cone {\mu}$ is a minimal graded free resolution of $R/I+(z)$ provided whenever $d_{ij} = c_{ij}$ for
some $i$ and $j$, the projection of the restriction map
$\mu _i|_{R(-d_{ij})}$
onto $R(-c_{ij})$ is identically zero.
If it is not zero, we can split off or cancel the
same two summands $R(- d_{ij})$  from both $F_i$ and $G_i$ in the
mapping cone construction to achieve minimality.

\begin{definition}
The {\it minimized mapping cone} of the map of complexes $\mu:{\bf F}\rar {\bf G}$, denoted by
$\mincone{\mu}$, is the complex obtained from $\cone{\mu}$ after splitting off all possible summands.
\end{definition}

If $z\in R$ is homogeneous element of degree $\delta$, $\bf F$ and $\bf G$ are minimal graded free resolutions of $R/I:z$ and $R/I$ respectively, and $\mu: {\bf F}(-\delta)\rar {\bf G}$ is the graded map of complexes  induced by multiplication by $z$, then
$\mincone{\mu}$ is a minimal graded free resolution of $R/I+(z)$.

\medskip

We finally mention the following result that we will need later on.

\begin {proposition}\label{sumcomplex}
Let
$${\bf F} : 0 \rar F_s\stackrel{f_s}{\lar} F_{s-1}
 \stackrel{f_{s-1}}{\lar} \cdots \lar F_t\stackrel{f_t}{\lar} \cdots \lar F_1
\stackrel{f_1}{\lar} F_0$$ and
$${\bf G }: 0 \rar G_s\stackrel{g_s}{\lar} G_{s-1}
 \stackrel{g_{s-1}}{\lar} \cdots \lar G_t\stackrel{g_t}{\lar} \cdots \lar G_1
\stackrel{g_1}{\lar} G_0$$ be two exact complexes of free modules and $\phi = \oplus
\phi_t : \bf F \to \bf G$ be a map of complexes.  Suppose that the
dual $\bf F^*$ is exact.  If $\phi_s = 0$ then we have a homotopy to
$\phi$ given by $\psi_i:F_{i-1}\to G_{i}$ for $1\le i\le s$ with
$\psi_s = 0$ and $\phi_i =  g_{i+1}\circ
\psi_{i+1} +(-1)^{s-i}\psi_i \circ f_i $. In particular,  $\phi_0(F_0) \subset   g_1(G_1)$.
\end{proposition}

\begin{proof}
Let $\psi_s:F_{s-1}\to G_s$ be the zero map.  Since,  $\phi_s= 0$, we get,  $\phi_{s-1}\circ f_s=0$ and hence $f_s^*(\phi _{s-1}^*)=0$.  By the exactness of the dual of ${\bf F}$, there exist a map $\psi_{s-1}^*:G_{s-1}^*\to F_{s-2}^*$ such that $f_{s-1}^*\circ \psi_{s-1}^*= \phi_{s-1}^*$.  So, we get, $\phi_{s-1} = \psi _{s-1}\circ f_{s-1}-g_s\circ \psi_s$.  Suppose that we have constructed $\psi_t$ such that, $\phi_i =  g_{i+1}\circ\psi{i+1} +(-1)^{s-i}\psi_i \circ f_i $ for all $i \ge  t$.    Now, since $\phi$ is a map of complexes, we have
$  \phi_{t-1}\circ f_t= g_{t}\circ \phi_{t}$.   Substituting, we get $ \phi_{t-1}\circ f_t =  g_{t}\circ (g_{t+1}\circ \psi{t+1}+(-1)^{s-t} \psi_{t}\circ f_{t}  )=  (-1)^{s-t}g_t \circ \psi_t \circ f_t$.  Thus $(\phi_{t-1}+(-1)^{s-t+1}g_t\circ \psi_t)\circ f_t=0$.  That is, $f_t^*((\phi_{t-1}^*+(-1)^{s-t+1} \psi_t^*\circ g_t^* )(G_{t-1}^*)=0$.
By the exactness of the dual ${\bf F^*}$, we get a map $\psi_{t-1}^*:G_{t-1}^*\to F_{t-2}^*$ such that $\phi_{t-1}^*+(-1)^{s-t+1} \psi_t^*\circ g_t^*  = f_{t-1}^*\circ \psi_{t-1}^*$ and hence
$\phi_{t-1}= g_t \circ \psi_t  +(-1)^{s-t+1} \psi_{t-1} \circ f_{t-1} $.  Thus we prove the existence of $\psi_i$ for all $i$.  Looking at $\psi_1:F_0\to G_1$, we get, $\phi_0(F_0) = g_1 \circ \psi_1(F_0) \subset g_1(G_1)$.

\end{proof}

\section{Gorenstein ideals}\label{secgorenstein}

In this section we deal with the case when $k[\Gamma(\mm)]$ is Gorenstein separately and see that an explicit construction of the minimal free resolution is obtained from one single mapping cone construction using the fact observed in Section~\ref{defidealsection} that $\fp=I_2(A)+I_2(B)$.
Note that the Gorenstein case will also fit into the general construction of iterated mapping cone given in Section~\ref{secmain} as we shall see in Remark~\ref{rmkGor}. The proof we provide in this section completes the argument presented in \cite[Theorem~2.6]{hip} for the case $n=4$.

\medskip

The explicit computation of the Cohen-Macaulay type of $k[\Gamma(\mm)]$ in
\cite[Corollary~6.2]{patseng} under the more general assumption of almost arithmetic sequences implies that if
$\mm$ is an arithmetic sequence then $k[\Gamma(\mm)]$ is Gorenstein if and only if $b=2$.

\medskip

So let's assume that $m_0 \equiv 2 \mod n$ and write $m_0 = an+2$. Then,
$$B= {\left(\begin{array}{cccc}
\begin{array}{c}
X_{n}^{a}\\[1.5mm]
X_{0}^{a+d}\\[1.5mm]
\end{array} &
\begin{array}{c}
X_{0}\\[1.5mm]
X_{2}\\[1.5mm]
\end{array} &
\begin{array}{c}
\cdots\\[1.5mm]
\cdots\\[1.5mm]
\end{array} &
\begin{array}{c}
X_{n-2}\\[1.5mm]
X_{n}\\[1.5mm]
\end{array}
\end{array}\right)}.$$
The ideals $I_2(A)$ and $I_2(B)$ are both of height $n-1$ and the interesting fact is that
the ideal $\fp = I_2(A)+I_2(B)$ is Gorenstein of height $n$.

\medskip

We will construct a resolution for $R/\fp$ as follows. We start with
the following preliminary lemma. Consider the map
\begin{equation}\label{alpha}
\alpha: D_{n-2}G^* \cong R^{n-1}\lar R\,,\quad
g_1^{(i)}g_2^{(n-2-i)}\mapsto ( -1)^i [
X_n^{a}X_{i+2}-X_0^{a+d}X_i]\,.
\end{equation}

\begin{lemma}\label{Alpha1}
Let ${\bf E}$ be the Eagon-Northcott complex which is a minimal
resolution of $R/I_2(A)$ and $\bf {E}^*$ be as in
Section~\ref{secENgraded}. Then the map $\alpha:\,E_{n-1}^* \lar R$
defined in {\rm(}\ref{alpha}{\rm)} induces a map of complexes
$\alpha:\ {\bf E}^*\to \bf E$.
\end{lemma}

\begin{proof}
Consider the basis element $e_t\otimes g_1^{(i)}g_2^{(n-3-i)}$ in
$\wedge ^{n-1}F^*\otimes D_{n-3}G^*$ where $e_t$ denotes
($(-1)^{n-i}  e_1^*\wedge \cdots e_{t-1}^*\wedge e_{t+1}^*\wedge
\cdots e_{n}^*$.  Then
$${\small
\begin{array}{rcl}
\alpha (d_{n-1}^* (e_i \otimes g_1^{(r)}g_2^{(s)}))
 & = & ( -1)^{i+1} X_{t-1}[ X_n^{a}X_{i+3}-X_0^{a+d}X_{i+1}]+(-1)^iX_t  [ X_n^{a}X_{i+2}-X_0^{a+d}X_i] \\
& = & (-1)^{i+1} (X_n^a [X_{t-1}X_{i+3}-X_{t }X_{i+2}] +X_0^{a+d}[X_{t-1}X_{i+1}-X_tX_{i}]) \\
& \in & I_2(A)\,.\\
\end{array}}$$
So, the composition $\wedge ^{n-1}F^*\otimes D_{n-3}G^* \lar
D_{n-2}G^* \stackrel{\alpha}{\lar}R \lar R/I_2(A)$ is zero.   Since
$E$ is exact, we can lift the map $\alpha$ to $\alpha: E^*\to E$ as
a map of complexes.
\end{proof}

\begin {theorem}{\label {minresgor}}
Let $\mm=(m_0, \ldots , m_n)$ be an arithmetic sequence with $m_0
\equiv 2 \mod n$. If $\fp\subset R$ is the defining ideal of the
monomial curve $C_\mm\subset\A_k^{n+1}$ associated to the sequence
$\mm$, then $R/\fp$ is Gorenstein of codimension $n$ with minimal
resolution given by ${\bf E} \oplus ({\bf E}^*)^{\bf 1}$ where $E$
is the Eagon-Northcott resolution of $R/I_2(A)$.
\end {theorem}

\begin{proof}
As we already recalled in Section~\ref{secprelim}, $\fp
=I_2(A)+I_2(B)$, and ${\bf E}\to R / I_2(A)\to 0$
is a minimal resolution of $R/I_2(A)$. Let ${\bf E}: 0\to R \to
E_1^* \to \cdots \to E_{n-2}^*\to E_{n-1}^*$ be its dual. If
$\alpha: R^{n-1} = E_{n-1}^* \to R$ is the map defined in
(\ref{alpha}), one has by Lemma~\ref{Alpha1} that $\alpha$ induces a
map of complexes ${\bf \alpha} :\bf E^* \lar \bf E$. Hence the
mapping cone ${\bf E} \oplus ({\bf E}^*)^{\bf 1}$ is exact.  Note
that the image of $\alpha$ is the ideal generated by the $n-1$
principal $2\times 2$ minors of $B$ and $H_0({\bf E} \oplus {\bf
E}^*) = \hbox{Im} (d_1)+\hbox{Im}(\alpha) = I_2(A)+I_2(B)$ because
the rest of the minors of $B$ are already in the ideal $I_2(A)$. So,
${\bf E}\oplus ({\bf E}^*)^{\bf 1}$ resolves $R/\fp$ and it is
minimal because $\alpha $ has positive degree.
\end{proof}

Since for $i\in[1,n-1]$, the rank of the free module $E_i$ is
$i{n\choose {i+1}}$, we immediately get the Betti numbers of
$R/\fp$:

\begin{corollary}\label{bettiGor}
Let $\mm = (m_0, \ldots , m_n)$ be an arithmetic sequence with $m_0
\equiv 2 \mod n$.  Then the Betti numbers of $R/\fp$, where
$\fp\subset\kx$ is the defining ideal of the monomial curve $C_\mm$
associated to the sequence $\mm$, are $\beta_0=1$, $\beta_i =
i{n\choose {i+1}}+(n-i){n\choose {i-1}}$ for $i\in [1,n-1]$, and
$\beta_n=1$.
\end{corollary}

\section{Explicit construction of a minimal graded resolution}\label{secmain}

Let's go back to the general situation: $m_0=an+b$ with $a,b$
positive integers and $b\in [1,n]$.  In this section, we will not really use that
$\fp = I_2(A)+I_2(B)$ as in the Gorenstein case but essentially that
$\fp$ is
minimally generated by the 2 by 2 minors of $A$ and the principal
minors of $B$, {\it i.e.}, $\Delta_i:=\Delta_{1,i-b+2}(B)=X_n^aX_{i}-X_0^{a+d}X_{i-b}$ for
$i\in [b,n]$. In other words, if one sets
$$
\forall\,i\in [b,n],\ I_i:=I_2(A)+ (\Delta _b, \ldots , \Delta _{i})\,,
$$
then $\fp=I_n$. For all $i\in [b,n]$, a graded free resolution of $R/I_i$
will be obtained by a series of iterated mapping cones.
Set
$$
\delta _i := \deg \Delta _{i} = m_0(a+d+1)+(i-b)d,\ \forall i\in
[b,n]
$$
(of course the degrees are with respect to the weighted grading).
The following lemma is key to the construction of the minimal resolutions.

\begin{lemma}\label {colon}
\begin{enumerate}
\item\label{firstcolon}
$I_2(A):\Delta _ b= I_2(A)$.
\item\label{coloninclusion}
$\forall\, i\in [b,n-1]$,
 $(X_0, X_1, \ldots , X_{n-1})\subseteq I_i:
\Delta _{i+1}$.
\end{enumerate}
\end{lemma}

\begin{proof}
(\ref{firstcolon}) holds because $I_2(A)$ is prime and
$\Delta_b\notin I_2(A)$. Moreover, for any $i\in [b,n-1]$ and $j\in
[0,n-1]$, one has that
$$
\begin{array}{rclcl}
X_j\Delta_{i+1}
&\equiv& X_n^aX_{i}X_{j+1}-X_0^{a+d}X_{i-b+1}X_j&\mod&
(X_jX_{i+1}-X_{j+1}X_i)\\
&\equiv& X_0^{a+d}X_{i-b}X_{j+1}-X_0^{a+d}X_{i-b+1}X_j&\mod&\Delta_i
\\
&=&X_0^{a+d}(X_{i-b}X_{j+1}-X_{i-b+1}X_j)&&
\end{array}
$$
which implies that $X_j\Delta_{i+1}\in I_2(A)+ (\Delta _{i})$
because $X_kX_{l+1}-X_{k+1}X_l\in I_2(A)$ for all $k,l\in [0,n-1]$, and we are done.
\end{proof}

\begin{remark}\label{rmkcolon}
Observe that Lemma~\ref{colon}~ (\ref{coloninclusion}) implies that, for all $i\in [b,n-1]$, either
$I_i:\Delta _{i+1}=(X_0, X_1, \ldots , X_{n-1})$ or $I_i:\Delta _{i+1}=(X_0, X_1, \ldots , X_{n-1}, X_n^\ell)$ 
for some $\ell\ge 1$. Indeed, we will see in (2) of Inductive Step~\ref{induction} that the latter never occurs.
\end{remark}

We are now ready for our iterated mapping cone construction.
Recall from Section~\ref{secENgraded} that the minimal graded free
resolution ${\bf E }=
\oplus _{i=0}^{n-1} E_i$ of $R/I_2(A)$ given by the Eagon-Northcott complex of the
matrix $A$ is
$$
0\rar E_{n-1}\lar\cdots\lar E_1\lar E_0\lar R/I_2(A)\rar 0
$$
with $E_0=R$ and
 $\displaystyle{E_{s-1}=
 \bigoplus_{k=1}^{s-1}
 \left(\bigoplus_{r\in\sigma(n,s) } R(- (sm_0 +kd
 +r d))\right)}$
for all $s\in [2,n]$.

\medskip

Let $\cc_b = {\bf E} \oplus {\bf E^1}(-\delta _b)$ denote the mapping cone of the map 
${\bf \Delta} _b: {\bf E} \rar {\bf E}$ which is induced by multiplication by 
$\Delta _b$ and $\delta_b = \deg(\Delta_b)$. By Lemma~\ref{colon}~(\ref{firstcolon}) 
together with the fact that all the individual maps in ${\bf \Delta} _b$ are multiplication 
by $\Delta _b$ and hence not zero (in fact injective), we get that $\cc_b$ is the 
minimal resolution of $R/I_ b$.

\begin{notation}\label{notationL}
Set
$
L(s,k):=\bigoplus_{r\in\sigma(n,s)}R(-[m_0(a+d+s+1)+kd+rd])
$
for all $s\in [1,n]$ and $k\ge 1$. Then, for all $s\in [2,n]$,
$(\cc_{b})_s=E_s\oplus\left(\bigoplus_{k=1}^{s-1}\left(L(s,k)\right)\right)$, and
the free modules in $\cc_b$ are
\begin{equation}\label{b}
\begin{array}{rcl}
(\cc_{b})_0&=&E_0=R,\\
(\cc_{b})_1&=&E_1\oplus E_{0}(-\delta_b)=E_1\oplus R(-(m_0(a+d+1))),\\
(\cc_{b})_s&=& E_s\oplus E_{s-1}(-\delta_b)\\
&=&\displaystyle{
E_s\oplus\left(\bigoplus_{k=1}^{s-1}\left(L(s,k)\right)\right),\ \forall s\in[2,n-1]
} \\
(\cc_{b})_n&=& \displaystyle{\bigoplus _{k=1}^ {n-1}
R(-[m_0(a+d+n+1)+kd+{n\choose 2}d])}\,.
\end{array}
\end{equation}
\end{notation}

\begin{remark}
If $b = n$, $\cc_b$ is the minimal resolution of
$R/\fp$. This is indeed the resolution obtained in
\cite[Theorem~3.4 and Corollary~3.5]{hip} in the case $b=n$.
\end{remark}

Let ${\bf K }= \oplus _{i=0}^n K_i$ be the Koszul  resolution of
$R/(X_0, \ldots , X_{n-1})$, i. e.,
$$
0\rar K_{n}\lar\cdots\lar K_1\lar K_0\lar R/(X_0,\ldots , X_{n-1})\rar
0
$$
with $K_0=R$,
$
K_1=\bigoplus_{k=0}^{n-1}R(-(m_0+kd))
$,
and $K_{s}=
\bigoplus_{r\in\sigma(n,s) }R(-(sm_0+r d))
$ for all $s\in [2,n]$.
Note that for all $s\in [2,n]$ and $i\in [b,n]$,
$$
K_s(-\delta_i)=K_s(-(m_0(a+d+1)+(i-b)d))=L(s,i-b)\,.
$$

For $i\in [b,n]$, we construct inductively two sequences of complexes ${\bf C}_{i}$ and ${\bf M}_i$
both resolving $R/I_i$ with ${\bf M}_i$ being a minimal resolution. For $i=b$,
${\bf C}_{b}={\bf M}_b={\bf E} \oplus {\bf E^1}(-\delta _b)$ is given in (\ref{b}).
We will now prove the following sequence of steps that forms the $i$-th step of our construction.

\begin{inductive}\label{induction}
\begin{enumerate}
\item
The minimal resolution of $R/I_{i-1}$ has length $n$.
\item
$I_{i-1}:\Delta _{i}=(X_0, X_1, \ldots , X_{n-1})$.
\item
Multiplication by $\Delta_i$ on $R$ induces a map of complexes
${\bf \Delta}_i:{\bf K}(-\delta_i)\rar {\bf M}_{i-1}$.
\item
${\bf C}_{i}=\cone{{\bf \Delta}_i}$ resolves $R/I_i$.
\item
${\bf M}_{i}$ is the minimized mapping cone of ${\bf \Delta}_i$ and is given by
\begin{itemize}
\item
$({\bf M}_{i})_0=R$,
\item
$({\bf M}_{i})_1=E_1\oplus
 \displaystyle{\bigoplus_{k=0}^{i-b}R(-[m_0(a+d+1)+kd])}$,
\item
$({\bf M}_{i})_s=
\left\{\begin{array}{ll}
E_s \oplus
 \left(\bigoplus_{k=s-1}^{i-b} L(s-1,k)\right)&
\hbox{ for }s\in [2,i-b+1],\\
E_{s} \oplus
\left(\bigoplus_{k=i-b+1}^{s-1} L(s,k)\right)&
\hbox{ for }s\in [i-b+2,n].
\end{array}\right.$
\end{itemize}
\item
The minimal resolution of $R/I_{i}$ has length $n$.
\end{enumerate}
\end{inductive}

\noindent
(1) $\Rightarrow$ (2).
As observed in Remark~\ref{rmkcolon}, by Lemma~\ref{colon}~ (\ref{coloninclusion}) one has that
$I_{i-1}:\Delta _{i}$ is either $(X_0, X_1, \ldots , X_{n-1})$ or $(X_0, X_1, \ldots , X_{n-1}, X_n^\ell)$ for some $\ell\ge 1$,
and it is well-known that the Koszul complex provides minimal graded free resolutions for $R/(X_0, \ldots , X_{n-1})$ and 
$R/(X_0, \ldots , X_{n-1}, X_n^\ell)$. Consider the Koszul  resolution ${\bf K'}= \oplus _{i=0}^{n+1}K_i'$  of
$R/(X_0, \ldots , X_{n-1}, X_n^\ell)$:
$$
0\rar K'_{n+1}\lar\cdots\lar K'_1\lar K'_0\lar R/(X_0, \ldots , X_{n-1}, X_n^\ell)\rar 0\,.
$$
Assume that $I_{i-1}:\Delta_{i}=(X_0, X_1, \ldots , X_{n-1}, X_n^\ell)$ for some $\ell\ge 1$ and consider the complex map ${\bf\Delta}'_{i}:{\bf K'}(-\delta_{i})\rar {\bf M}_{i-1}$ induced by multiplication by $\Delta_{i}$. Then the mapping cone of ${\bf\Delta}'_{i}$ provides a free resolution of $R/I_{i}$ and since ${\bf K'}(-\delta_{i})$ and ${\bf M}_{i-1}$ have length $n+1$ and $n$ respectively, $\cone{{\bf\Delta}'_{i}}$ has length $n+2$. It may not be minimal nevertheless, since $({\bf\Delta}'_{i})_{n+1}:K'_{n+1}(-\delta_{i})\rar 0$ is the zero map, no cancelation can occur at the last step of $\cone{{\bf\Delta}'_{i}}$ and $R/I_{i}$ would have a minimal resolution of length $n+2$ which is impossible since
$R$ is a regular ring of length $n+1$. Thus, $I_{i-1}:\Delta _{i}=(X_0, X_1, \ldots , X_{n-1})$.

\medskip\noindent
(2) $\Rightarrow$ (3) $\Rightarrow$ (4) and (5) $\Rightarrow$ (6) are straightforward. It remains to show that (4) $\Rightarrow$ (5).

\medskip\noindent
Set $i=b+t$. The complex map ${\bf \Delta} _{b+t}: {\bf K}(-\delta
_{b+t})\rar {\bf M}_{b+t-1}$
induced by multiplication by
$\Delta_{b+t}$ (which is of degree $\delta_{b+t}=m_0(a+d+1)+td$) is
given by the following diagram (the left column is the shifted
Koszul complex ${\bf K}(-\delta_{b+t})$ that resolves
$R(-[m_0(a+d+1)+td])/(X_0, \ldots , X_{n-1})$ minimally,
and the column on the right hand side is ${\bf M}_{b+t-1}$ that resolves $R/I_{b+t-1}$ minimally):

 \bigskip

$$
\begin{CD}
0 @. 0\\
@VV  V   @VV   V\\
L(n,t)
 @> ({\bf \Delta}_{b+t})_n>>
\bigoplus_{k=t}^ {n-1} L(n,k)
 \\
@VV V    @VV V\\
L(n-1,t)
 @>({\bf \Delta}_{b+t})_{n-1}>>
E_{n-1}\oplus
\left(\bigoplus_ {k={t}}^{n-2}
 L(n-1,k)\right)
 \\
@VV V    @VV V\\
\vdots @. \vdots \\
@VV V    @VV V\\
L(t+2,t)
 @> ({\bf \Delta} _{b+t})_{t+2}>>
E_{t+2} \oplus
\left(\bigoplus_{k=t}^{t+1} L(t+2,k)\right)
 \\
@VV V    @VV V\\
L(t+1,t)
 @> ({\bf \Delta} _{b+t})_{t+1}>>
E_{t+1} \oplus
 L(t+1,t)
 \\
@VV V @ VV V\\
L(t,t)
 @>({\bf\Delta} _{b+t})_{t}>>
E_{t}\oplus L(t-1,t-1)
\\
@VV V @ VV V\\
\vdots @. \vdots \\
@VV V @ VV V\\
L(3,t)
 @>({\bf\Delta}_{b+t})_3>>
E_3 \oplus
 \left(\bigoplus_{k=2}^{t-1} L(2,k)\right)
 \\
@VV V @ VV V\\
L(2,t)
 @>({\bf \Delta} _{b+t})_{2}>>
E_{2} \oplus
\left(\bigoplus_{k=1}^{t-1}
L(1,k) \right)
 \\
@VV V @ VV V\\
L(1,t)
 @>({\bf \Delta}_{b+t})_1>>
 E_1\oplus
 \bigoplus_{k=0}^{t-1}R(-[m_0(a+d+1)+kd])
 \\
@VV V @ VV V\\
R(-[m_0(a+d+1)+td])
 @>
({\bf \Delta}_{b+t})_0
 >>
R
\end{CD}
$$

\medskip
Now observe in the previous diagram that, for $s\in [t+1,n]$, the left side $({\bf
K}(-\delta_{b+t}))_s$ corresponds to the summand $k=t$ on the
right hand side and we will show that it splits off entirely.

\medskip
The following lemma shows that none of these maps are identically zero.

\begin {lemma}\label {positive}
Let ${\bf K}$, ${\bf M}$ and ${\bf \Delta} _{b+t}: {\bf K}(-\delta
_{b+t})\rar {\bf M}_{b+t-1}$ be as before. Then, $({\bf \Delta} _{b+t})_i\neq 0$ for all
 $0\le i\le n$. In fact, if we
choose a basis for the free modules $K_i$,  then we can pick a map
of complexes ${\bf \Delta} _{b+t}$ such that $({\bf \Delta} _{b+t})_i $ is not zero on any of the
chosen basis elements of $K_i$.
\end{lemma}

\begin{proof}
Since $\bf M$ is minimal of length $n$ and $({\bf \Delta} _{b+t})_0$ given by
multiplication by $\Delta_{b+t}$ is injective (not zero),  all the
maps $({\bf \Delta} _{b+t})_i, 0\le i\le {n-1}$ are not zero and can be so chosen to
be not zero on any chosen basis elements of $K_t$.  The only
question is for the last map $({\bf \Delta} _{b+t})_n$.  This we will take care by the
Proposition~\ref{sumcomplex}. Since
$\Delta _{b+t}$ is not contained in the ideal $I_{b+t-1}$, we get that
$({\bf \Delta} _{b+t})_0(R) = \Delta _{b+t}R$ is not contained in the image of ${\bf M}_1$.
Thus by Proposition~\ref{sumcomplex}, we see that $({\bf \Delta} _{b+t})_n$ is not
equal to zero. Since $K_n= R$, therefore, we get that $({\bf \Delta} _{b+t})_n$ is not
equal to zero on a basis for the free module $K_n$.
\end{proof}

\medskip

Next we show that the projection of $({\bf \Delta} _{b+t})_s$ onto $L(s,t)$ is not zero for any $s\geq t+1$.
By degree consideration, the projection of $({\bf \Delta} _{b+t})_s$ onto $L(s,k)$ for $k>t$ is certainly zero. What is left follows from the next lemma.

\begin {lemma}\label {projection}
For all $1\le t\le n-b$,
the projection of $({\bf \Delta}_{b+t})_k(R(-[m_0(a+d+k+1)+td+rd])$ onto
$R(-[m_0(a+d+k+1)+td+rd])$ is not zero, and hence is a multiplication by a unit, 
for every $k\in [t+1,n]$ and $r\in \sigma(n,k)$.
\end{lemma}

\begin{proof}
It suffices to show that
\begin{equation}\label{p}
({\bf \Delta}_{b+t})_k(R(-[m_0(a+d+k+1)+td+rd])\not\subset E_k
\end{equation}
that is, the projection of $({\bf \Delta}_{b+t})_k(L(k,t))$ onto
$L(k,t)$ is not zero. If this projection is not zero for some $k$ and $t$, then the projection of $({\bf \Delta}_{b+t})_k(R(-[m_0(a+d+k+1)+td+rd])$ onto $R(-[m_0(a+d+k+1)+td+rd]$
is not zero for the lowest $r\in\sigma(n,k)$. Then we can split it off and go to the next smallest $r$ and so we get the lemma.

\smallskip
Now the rest of the proof is by descending induction on $k$.
None of these maps $\Delta$ are identically zero by Lemma~\ref{positive}.
If $k=n$, $E_n=0$ and (\ref{p}) holds, hence the lemma is true for all $t$.

\smallskip
Assume that the lemma holds for all $k\in [s+1,n]$ and for some $s$.  Let $k=s$.  Suppose that there 
is an $r= r_1+\cdots +r_s$ such that $(\Delta _{b+t})_{s}(R(-[m_0(a+d+s+1)+td+rd]) $ is
entirely in $E_s$.  Pick $r_0$ to be the smallest non negative integer not in $\{r_1, \ldots,
r_s\}$. Consider the commutative diagram:
$$
\begin{array}{ccc}
  R(-[m_0(a+d+s+2)+td+r+r_0d] &\lar&  ({\bf  M}_{b+t-1})_{s+1}\\
     \downarrow     &&                  \downarrow\\
L(s,t) &\lar            & ({\bf M}_{b+t-1})_s
\end{array}
$$

Since the lemma is true for $s+1$, we can take for every $r\in \sigma(n,s+1)$,
$$(\Delta _{b+t})_{s+1}(R(-[m_0(a+d+s+2)+td+rd+r_0d])  =
R(-[m_0(a+d+s+2)+td+rd+r_0d])+...$$
Continuing with the vertical arrow on the right, we see that
$R(-[m_0(a+d+s+2)+td+rd+r_0d])$
 maps onto
$$\Delta _{b+t-1}E_S+\sum _{i\ge
0}(\pm)X_{r_i}R(-[m_0(a+d+s+2)+td+(r+r_0-r_i)d]).$$
Thus
every one of the $r_i$ and in particular $X_{r_0}$
that makes up the sum $r+r_0$ will appear as part of the the image in $(M_{b+t-1})_s/E_s$.  Thus this
image is not
contained in $E_s\bigoplus (X_{r_1}, X_{r_2},\ldots , X_{r_s})\bigoplus
R(-[m_0(a+d+s+2)+(t-1)d+(r+r_0-r_i)d])$.
If we first come down
and then apply $(\Delta _{b+t-1})_s$,  the image is
$(\Delta_{b+t-1})_s(\sum _{i\ge 0} \pm X_{r_i} 1 (-[m_0(a+d+s+2)+td+(r+r_0-r_i)d])$ which 
is contained in $(E_s\bigoplus (X_{r_1}, X_{r_2},\ldots , X_{r_s})\bigoplus
R(-[m_0(a+d+s+2)+td+(r+r_0-r_i)d])$ - a contradiction. This completes the induction and the proof. 
\end{proof}

For $s\in [t+1,n]$, the left side $({\bf
K}(-\delta_{b+t}))_s$ splits off entirely with the summand $k=t$ on the
right hand side.
On the other
hand, for $s\in [0,t]$, no cancelation occurs. So the minimal free resolution of $R/I_i$ is as described in (5) of Inductive Step~\ref{induction}. This completes our inductive construction and we have proved the following:

\begin{theorem}\label{resIi}
Let $\mm=(m_0, \ldots, m_n)$ be an arithmetic sequence with common difference $d$ and write
$m_0=an+b$ for $a,b$ two positive integers with $b\in [1,n]$.
Consider the polynomial ring $R=k[X_0,\ldots , X_n]$ with $\deg{X_i}=m_i$. Let $J\subset R$ be the defining ideal
of the rational normal curve and ${\bf E}$ be its minimal resolution given by the Eagon-Northcott complex. Set $\Delta_i:=X_n^aX_{i}-X_0^{a+d}X_{i-b}$ for all
$i\in [b,n]$, $\delta_i:=\deg{\Delta_i}=m_0(a+d+1)+(i-b)d$, and consider the ideal
$I_i:=J+ (\Delta _b, \ldots, \Delta _{i})\subset R$. Then for all $i\in [b,n]$, $R/I_i$ is Cohen-Macaulay of codimension $n$ with
minimal graded free resolution ${\bf M}_i$ given by
$$
\begin{array}{rcl}
{\bf M}_b&=&\mincone{{\bf \Delta}_b:{\bf E}(-\delta_b)\rar {\bf E}}=\cone{{\bf \Delta}_b:{\bf E}(-\delta_b)\rar {\bf E}}
\\
{\bf M}_i&=&\mincone{{\bf \Delta}_i:{\bf K}(-\delta_i)\rar {\bf M}_{i-1}}\,,\ \forall i\in [b+1,n]
\end{array}
$$
where ${\bf K}$ is the Koszul resolution of $(X_0,\ldots,X_{n-1})$.
The free modules in ${\bf M}_i$ are explicitly given by
$$
0\rar  ({\bf M}_i)_n \lar ({\bf M}_i)_{n-1} \lar \cdots \lar ({\bf M}_i)_1\lar R\lar R/I_i\rar 0
$$
where
\begin{itemize}
\item
$({\bf M}_{i})_0=R$,
\item
$({\bf M}_{i})_1=E_1\oplus
 \displaystyle{\bigoplus_{k=0}^{i-b}R(-[m_0(a+d+1)+kd])}$,
\item
$({\bf M}_{i})_s=
\left\{\begin{array}{ll}
E_s \oplus
 \left(\bigoplus_{k=s-1}^{i-b} L(s-1,k)\right)&
\hbox{ for }s\in [2,i-b+1],\\
E_{s} \oplus
\left(\bigoplus_{k=i-b+1}^{s-1} L(s,k)\right)&
\hbox{ for }s\in [i-b+2,n].
\end{array}\right.$
\end{itemize}
In particular,
$({\bf M}_{i})_n=
\left\{\begin{array}{ll}
L(n-1,n-1)&
\hbox{ if $i=n$ and $b=1$},\\
\left(\bigoplus_{k=i-b+1}^{n-1} L(n,k)\right)&
\hbox{ otherwise.}
\end{array}\right.$
\end{theorem}

\begin{corollary}\label{CMtypeIi}
Using notations in Theorem~\ref{resIi}, the Cohen-Macaulay of type of $R/I_i$ is
$\left\{
\begin{array}{ll}
n &\hbox{ if } i=n \hbox{ and } b=1,\\
n-1+b-i &\hbox{ otherwise.}
\end{array}\right.$
\end{corollary}

In particular, since $I_n=\fp$ is the defining ideal of the monomial curve $C_\mm$, we get the following theorem.

\begin{theorem}\label{mainThm}
Let $\mm=(m_0, \ldots, m_n)$ be an arithmetic sequence and write
$m_0=an+b$ for $a,b$ two positive integers with $b\in [1,n]$. If
$\fp\subset R:=\kx$ is the defining ideal of the monomial curve
$C_\mm\subset\A_k^{n+1}$ associated to $\mm$ then $R/\fp$ is
Cohen-Macaulay of codimension $n$ and its minimal graded free resolution ${\bf M}_n$
is
$$
0\rar  F_n \lar E_{n-1}\oplus F_{n-1} \lar \cdots \lar E_1\oplus
F_1\lar R\lar R/\fp\rar 0
$$
where, for all $s\in[2,n]$, $\displaystyle{E_{s-1}=
 \bigoplus_{k=1}^{s-1}
 \left(\bigoplus_{r\in\sigma(n,s) } R(- (sm_0+kd+rd))\right)}$, and
$$
\begin{array}{rcl}
F_1  &=& \displaystyle{ \left(\bigoplus _{k=0}^{n-b}R(-[m_0(a+d+1)+kd])\right)}\\
F_2  &=& \displaystyle{ \left(\bigoplus_{k=1}^{n-b}\left(
\bigoplus_{r=0}^{n-1}R(-[(m_0(a+d+2)+kd +rd])\right)\right)}
\\
F_s  &=&  \left\{
 \begin{array}{ll}
 \displaystyle{
\left(\bigoplus_{k=0}^{n-b+1-s}\left( \bigoplus_{r\in \sigma (n,s-1)}
R(-[m_0(a+d+s)+kd+rd])\right)\right)}& \hbox{\rm if}\ s\in[3,n-b+1],
 \\
 \displaystyle{
\left(\bigoplus_{k=1}^{s-n+b-1}\left( \bigoplus_{r\in \sigma (n,s)}
R(-[m_0(a+d+s+1)+kd+rd])\right)\right)}& \hbox{\rm if}\
s\in[n-b+2,n].
 \end{array}\right.
\end{array}
$$
\end{theorem}

\section{Consequences}\label{secconsequences}

The first direct consequence of Theorem~\ref{mainThm} shows that
\cite[Conjecture~2.2]{hip}, which was first stated in \cite{eaca},
holds:

\begin{theorem}\label{ThmIndraConj}
With notations as in Theorem~\ref{mainThm}, the Betti numbers of
$R/\fp$ only depend on $n$ and on the value of $m_0$ modulo $n$.
\end{theorem}

\begin{proof}
If one reads the Betti numbers $\{\beta_j,\ j\in [0,n]\}$ in the
minimal graded free resolution in Theorem~\ref{mainThm}, one gets
that $\beta_0=1$ and
\begin{equation}\label{bettiGeneral}
\beta_j=j{n\choose j+1}+
 \left\{\begin{array}{ll}
\displaystyle{(n-b+2-j){n\choose j-1}}&\hbox{if}\ 1\leq j\leq n-b+1,\\
\displaystyle{(j-n+b-1){n\choose j}}&\hbox{if}\ n-b+1<j\leq n,
 \end{array}\right.
\end{equation}
where $m_0\equiv b\mod n$ and $b\in [1,n]$ and hence the statement
holds.
\end{proof}

\begin{example}
For $n=4$, if $\fp\subset R=k[X_0,\ldots , X_4]$ is the defining ideal of the monomial curve associated
to an arithmetic sequence $\mm=(m_0,\ldots , m_4)$, the $4$ different patterns for the global Betti numbers of
$R/\fp$ are as follows. They correspond respectively to $b=1$, 2, 3 and 4:
$$
\begin{array}{lllllllllllllll}
0 & \rar &  R^4 & \lar & R^{15} & \lar & R^{20} & \lar & R^{10} & \lar & R & \lar & R/\fp & \rar & 0 \\
0 & \rar &  R      & \lar & R^9      & \lar & R^{16} & \lar & R^9      & \lar & R & \lar & R/\fp & \rar & 0 \\
0 & \rar &  R^2 & \lar & R^7      & \lar & R^{12} & \lar & R^8      & \lar & R & \lar & R/\fp & \rar & 0 \\
0 & \rar &  R^3 & \lar & R^{11} & \lar & R^{14} & \lar & R^7      & \lar & R & \lar & R/\fp & \rar & 0
\end{array}
$$
For example, the two arithmetic sequences $\mm_1=(11,13,15,17,19)$ and $\mm_2=(7,12,17,22,27)$ fit into the third pattern because $11\equiv 7\equiv 3\mod 4$. Denoting by $\fp_1$ and $\fp_2$ the defining ideal of $C_{\mm_1}$ and $C_{\mm_2}$ respectively, the minimal graded free resolutions of $R/\fp_1$ and $R/\fp_2$ are given by Theorem~\ref{mainThm}, and the result can easily be checked
using the softwares CoCoA, Macaulay2 or Singular:
{\footnotesize
\begin{eqnarray*}
0\rar
R(-115)\oplus R(-117)
\lar
\begin{array}{r}
R(-58)\oplus R(-60)\\
\oplus R(-62)\oplus R(-98)\\
\oplus R(-100)\oplus R(-102)\\
\oplus R(-104)
\end{array}
\lar
\begin{array}{r}
R(-41)\oplus R(-43)^2\\
\oplus R(-45)^2\oplus R(-47)^2\\
\oplus R(-49)\oplus R(-68)\\
\oplus R(-70)\oplus R(-72)\\
\oplus R(-74)
\end{array}\\
\lar
\begin{array}{r}
R(-26)\oplus R(-28)\\
\oplus R(-30)^2\oplus R(-32)\\
\oplus R(-34)\oplus R(-55)\\
\oplus R(-57)
\end{array}
\lar
R\lar R/\fp_1\rar 0.
\end{eqnarray*}
\begin{eqnarray*}
0\rar
R(-117)\oplus R(-122)
\lar
\begin{array}{r}
R(-63)\oplus R(-68)\\
\oplus R(-73)\oplus R(-95)\\
\oplus R(-100)\oplus R(-105)\\
\oplus R(-110)
\end{array}
\lar
\begin{array}{r}
R(-41)\oplus R(-46)^2\\
\oplus R(-51)^2\oplus R(-56)^2\\
\oplus R(-61)^2\oplus R(-66)\\
\oplus R(-71)\oplus R(-76)
\end{array}\\
\lar
\begin{array}{r}
R(-24)\oplus R(-29)\\
\oplus R(-34)^2\oplus R(-39)\\
\oplus R(-44)\oplus R(-49)\\
\oplus R(-54)
\end{array}
\lar
R\lar R/\fp_2\rar 0.
\end{eqnarray*}
}
\end{example}

\begin{remark}
It is important to note  that the phenomenon described in
Theorem~\ref{ThmIndraConj} is something special about arithmetic
sequences only. In general, if $0 < m_{0} < \cdots < m_{n}$ is a sequence of
integers then the Betti numbers of the monomial curve defined by
$x_{i} = t^{m_{i}}$ do not depend only on $n$ and the remainder of
$m_{0}$ upon division by $n$.  It does not hold even for almost
arithmetic sequences, even in dimension 3. 
In dimension 3, a monomial curve is either a complete
intersection with $\beta_1=2, \beta_2=1$ or an ideal of $ 2\times 2$ minors
of a $3\times 2$ matrix with $\beta_1=3$ and $\beta_2=2$.   Now, it is easy
to see that for $\mm = (7, 10, 15)$,  $C_\mm$ is a complete intersection with
$\fp = (X_1^3-X_2^2, X_0^5-X_1^2X_2)$ where as for $\mm =(13,16, 21)$, $C_\mm$ is
not a complete intersection.  However both $7$ and $13 $ are odd and
hence equal 1 modulo 2.
 \end {remark}

\begin{corollary}
If $I_i\subset R$ is as in Theorem~\ref{resIi}, $R/I_i$ is Gorenstein if and only if
\begin{itemize}
\item
$b=2$, $i=n$,
\item
$b=1$, $i=n-1$, or
\item
$n=1$.
\end{itemize}
In particular, we recover the result of \cite[Corollary~6.2]{patseng} that $R/\fp$ is Gorenstein if and only if $b=2$.
Moreover, $R/I_i$ are never level unless they are Gorenstein.
\end{corollary}

\begin{proof}
If $n=1$, $I_1=\fp$ is principal and therefore Gorenstein. By Corollary~\ref{CMtypeIi}, the type of $R/I_i$ is $n-1+b-i$. Since $i\neq n$,
$n-1+b-i\geq b-1>1$ if $b>2$ so $R/I_i$ is Gorenstein if and only if either $b=2$, $i=n$ or $b=1$, $i=n-1$.
The non-levelness of $R/I_i$ when it is not Gorenstein follows directly from the degrees in the resolution.
\end{proof}

\begin{remark}\label{rmkGor}
If $m_0\equiv 2\mod n$, i.e., if $R/\fp$ is Gorenstein, then by
(\ref{bettiGeneral}) one has that $\beta_0=\beta_n=1$ and
$\beta_j=j{n\choose j+1}+(n-j){n\choose j-1}$ for all $j\in [1,n-1]$
which is Corollary~\ref{bettiGor}. Note that this result is obtained
here by making a resolution minimal which is obtained through iterated 
mapping cone construction, while it had been obtained in 
Section~\ref{secgorenstein} by a direct argument.
\end{remark}

\begin{remark}
If $m_0\equiv 1\mod n$, one gets by (\ref{bettiGeneral}) that for
all $j\in [1,n]$, $\beta_j=j{n\choose j+1}+(n+1-j){n\choose j-1}$.
Note that in this case the Betti numbers had already been obtained
in \cite[Theorem~3.1]{hip} where we show that for all $j\in [1,n]$,
$\beta_j=j{n+1\choose j+1}$. One can easily check that both numbers
coincide.
\end{remark}

\begin{theorem}\label{CMtype}
With notations as in Theorem~\ref{mainThm}, the Cohen
Macaulay type of $R/\fp$ is the unique integer $c$ in $[1,n]$
such that $c\equiv m_0-1 \mod n$.
\end{theorem}

\begin{proof}
The Cohen-Macaulay type of $R/\fp$, $\beta_n$, is computed by the
first formula in (\ref{bettiGeneral}) when $b=1$, and by the second
otherwise. Thus, $\beta_n={n\choose n-1}=n$ if $b=1$, and
$\beta_n=(b-1){n\choose n}=b-1$ otherwise.
\end{proof}

Putting together Theorem~\ref{ThmIndraConj} and
Theorem~\ref{CMtype}, one gets the following:

\begin{corollary}\label{CMtypeAndBetti}
With notations as in Theorem~\ref{mainThm}, the Cohen-Macaulay type
of $R/\fp$ determines all its Betti numbers.
\end{corollary}

Note that in the previous corollary, the same result holds if one substitutes the minimal number of generators of $\fp$ for the Cohen-Macaulay type of $R/\fp$.

\begin {remark}
One can also deduce from the minimal graded free resolution in Theorem~\ref{mainThm}, the value of the (weighted) Castelnuovo-Mumford regularity of $R/\fp$ which is indeed
$$
\hbox{reg}(R/\fp)=
\left\{\begin{array}{ll}
{n\choose 2}d+m_0(a+d)+n(m_0-1)&\hbox{if }b=1,
\\
({n\choose 2}+b-1)d+m_0(a+d+1)+n(m_0-1)&\hbox{if }b\geq 2.
\end{array}\right.
$$
On the other hand, the Frobenius number $g(\mm)$ of the numerical semigroup $\Gamma(\mm)$ can be computed  using \cite[Theorem~3.2.2]{ramirezbook} and one gets that $g(\mm)=(a-1)m_0+d(m_0-1)$ if $b=1$, and
$g(\mm)=am_0+d(m_0-1)$ if $b\geq 2$. Thus
$
\hbox{reg}(R/\fp)-g(\mm)=({n\choose 2}+b)d+m_0+n(m_0-1)\,.
$
In particular, the regularity is always an upper bound for the conductor $g(\mm)+1$ of the numerical semigroup $\Gamma(\mm)$ and it is, in general, much bigger. The above relation between the regularity of the semigroup ring and the Frobenius number of the semigroup 
can be nicely expressed as follows:
$$
\hbox{reg}(R/\fp)=g(\mm)
+\sum_{i=0}^{n}m_i-(n-b)d-n\,.
$$
\end{remark}

\medskip

Finally, we use Theorem \ref{ThmIndraConj} to prove a conjecture of Herzog and
Srinivasan on eventual periodicity of Betti numbers of semigroup
rings in our context.
Given a sequence of positive integers $\mm= (m_0, \ldots, m_n)$ and a positive integer $j$,
denote by $\mm+(j) = \mm +(j,j, \ldots, j)$. Herzog and
Srinivasan have conjectured the following:
\begin{Conjecture}[Herzog and Srinivasan] Let $\mm$ and $\mm+(j)$ be as above.
\begin{itemize}
\item[HS1] The Betti numbers of the semigroup ring $k[\Gamma(\mm +(j))]$
are eventually periodic in $j$.
\item[HS2] The number of minimal generators of the defining ideal of the monomial curve $C_{\mm+(j)}$
is eventually periodic in $j$ with period $m_n-m_0$.
\item[HS3] The number of minimal generators of the defining ideal of the monomial curve $C_{\mm+(j)}$
is bounded for all $j$.
\end{itemize}
\end{Conjecture}
 They prove the conjecture for $n=2$
and A. Marzullo proves it for some cases when $n=4$ in \cite {Ad}.
Our Theorem~\ref{ThmPeriodic} proves this periodicity
conjecture in its strong form (HS1) for arithmetic sequences.

\begin{theorem}\label{ThmPeriodic}
If $\mm=(m_0,\ldots,m_n)$ is an arithmetic sequence and $\mm+(j)=\mm +(j,j, \ldots, j)$, then the Betti numbers of $C_{\mm+(j)}$ are eventually periodic in $j$ with period $m_n-m_0$.
\end{theorem}

\begin{proof}
Let $\mm=(m_0,\ldots,m_n)$ be an arithmetic sequence and $j\in\N$. Since $(m_0+j, \ldots, m_n+j)$ is in arithmetic progression, $\gcd{(m_0+j,\ldots,m_n+j)}=\gcd{(m_0+j,d)}$ where $d$ is the common difference in $\mm$.
We denote by $\widetilde{.}$ division by $\gcd{(m_0+j,d)}$. Then $\widetilde{m_i+j}=\widetilde{m_0+j}+i\widetilde{d}$ and
$\widetilde{\mm+(j)}=(\widetilde{m_0+j},\ldots,\widetilde{m_n+j})$ has gcd 1. Moreover, $C_{\mm+(j)}=C_{\widetilde{\mm+(j)}}$.

\smallskip
We claim that for $j\geq nd-m_0$, $\widetilde{\mm+(j)}$ is always an arithmetic sequence. If
$\widetilde{m_k+j}=\sum_{i\neq k}r_i(\widetilde{m_i+j})$ for $r_i\in\N$, then
$k\widetilde{d}+\widetilde{m_0+j}=\sum_{i\neq k}r_i(\widetilde{m_0+j}+i\widetilde{d})$ and hence
$\widetilde{d}(k-\sum_{i\neq k}ir_i)=(\widetilde{m_0+j})(\sum_{i\neq k}r_i-1)\geq n\widetilde{d}$ and this is not possible since $k-\sum_{i\neq k}ir_i<n$.
Now,
\begin{eqnarray*}
\widetilde{m_0+j+nd}&=&\widetilde{m_0+j}+n\widetilde{d}\\
&\equiv&\widetilde{m_0+j}\mod n.
\end{eqnarray*}
Note that $\gcd{(m_0+j,d)}$ is periodic with period $d$. So the Betti numbers of $C_{\mm+(j)}$ are periodic with period $nd$ for $j$ large enough.
\end{proof}

\bigskip

\bibliographystyle{amsplain}

\bigskip\bigskip

\end{document}